\documentclass[10pt]{article}
\usepackage{graphicx}
\usepackage{amsmath}
\usepackage{amssymb}
\usepackage{amscd}
\usepackage{hhline}
\usepackage{epic}
\usepackage{mathrsfs}

\newtheorem{theorem}{Theorem}
\newtheorem{theorema}{Theorem}
\newtheorem{theoremb}{Theorem}
\newtheorem{theoremc}{Theorem}
\newtheorem{theoremd}{Theorem}
\newtheorem{theoreme}{Theorem}
\newtheorem{dfn}[theoremb]{Definition}
\newtheorem{rk}[theoremc]{Remark}
\newtheorem{cor}[theoremd]{Corollary}
\newtheorem{lem}[theoreme]{Lemma}

\newenvironment{proof}[1][Proof]{\textbf{#1.} }{\qed \vspace{5pt}}
\newcommand\bib[1]{\bibitem[#1]{#1}}

\newcommand\abz{\hspace{13.5pt}}
\newcommand\qed{\phantom{\underline{y}}\hfill\hfill$\square$}
\newcommand{\comm}[1]{}

\newcommand\1{{\bf 1}}
\renewcommand\a{\alpha}
\renewcommand\b{\beta}

\newcommand\C{{\mathbb C}}

\newcommand\e{\varepsilon}
\newcommand\E{{\mathcal E}}

\newcommand\g{\mathfrak{g}}

\newcommand\h{\mathfrak{h}}

\renewcommand\l{\lambda}

\newcommand\La{\Lambda}
\newcommand\m{\mathfrak{m}}

\newcommand\oo{\omega}
\newcommand\op[1]{\mathop{\rm #1}\nolimits}
\newcommand\ot{\otimes}
\newcommand\p{\partial}

\newcommand\R{{\mathbb R}}

\newcommand\vp{\varphi}
\newcommand\we{\wedge}
\newcommand\x{\xi}
\newcommand\z{\sigma}
\newcommand\Z{{\mathbb Z}}

\makeatletter
\renewcommand{\@evenhead}{\hfil Boris Kruglikov \hfil}
\renewcommand{\@oddhead}{\hfil Symmetries of  almost complex structures\hfil}
\renewcommand{\@begintheorem}[2]{\begin{trivlist}\it
 \item[\hspace{\labelsep}{\bf #1\ #2.}]}
\renewcommand{\@endtheorem}{\end{trivlist}}
\makeatother

\begin{document}

 \title{Symmetries of  almost complex structures \\ and pseudoholomorphic foliations}
 \author{Boris Kruglikov}
 \date{}
 \maketitle

 \vspace{-14.5pt}
 \begin{abstract}
Contrary to complex structures, a generic almost complex structure $J$ has no local symmetries.
We give a criterion for finite-dimensionality of the pseudogroup $G$ of local symmetries for
a given almost complex structure $J$.
It will be indicated that a large symmetry pseudogroup (infinite-dimensional) is a signature of some
integrable structure, like a pseudoholomorphic foliation. We classify the sub-maximal
(from the viewpoint of the size of $G$) symmetric structures $J$.
Almost complex structures in dimensions 4 and 6 are discussed in greater details in this paper.%
 \footnote{MSC numbers: 32Q60, 53C15; 53A55, 22E65.\\ Keywords:
almost complex structure, symmetry, equivalence, differential invariant, Nijenhuis tensor,
pseudoholomorphic foliation.}%
 \end{abstract}

\section*{Introduction and Results}

 \abz
Let $(M^{2n},J)$ be a connected smooth almost complex manifold, $J^2=-\bf1$. In this paper we focuss
mainly on local geometry and assume $n=\frac12\dim M>1$. In this case the Nijenhuis tensor
$N_J(\x,\eta)=[J\x,J\eta]-J[\x,J\eta]-J[J\x,\eta]-[\x,\eta]$
(which is a skew-symmetric (2,1)-tensor) is generically non-zero.

It is known that all differential invariants of $J$ can be expressed via the jets of the
Nijenhuis tensor \cite{K$_1$}.
In this paper we will be concerned with vector distributions canonically associated with $N_J$.
We will see that their behavior is related to the amount of local symmetries of the almost complex structure.

Of course, the most symmetric case corresponds to $N_J=0$, and this vanishing is equivalent to local integrability
of $J$ \cite{NN}. On the other pole we know that generically $J$ has no local symmetries at all \cite{K$_2$}.
What happens in between?

Globally, if the manifold $M$ is closed, the symmetry group $\op{Aut}(M,J)$ is a finite-dimensional
Lie group \cite{BKW}. For the pseudogroup of local symmetries $G=\op{Sym}_\text{\rm loc}(M,J)$ we
will prove the following criterion of finite-dimensionality.

 \begin{theorema}\label{thA}
Suppose that the almost complex structure $J$ is non-degenerate in the sense that the
Nijenhuis tensor as the map $N_J:\La^2TM\to TM$ is epimorphic for $n>2$ and that
its image $\Pi=\op{Im}(N_J)\subset TM$ is a non-integrable vector distribution for $n=2$.
Then $\dim G<\infty$.
 \end{theorema}

In general the pseudogroup of local symmetries $G=\op{Sym}_\text{\rm loc}(M,J)$ can
be infinite-dimensional even in the almost complex case.
By the Cartan test \cite{BCG$^3$,KL$_2$,Se} we can calculate the functional dimension and rank
of $G$, i.e. determine that its elements formally depend on $\z$ functions of $p$ arguments
(and some number of functions with fewer arguments).
For instance, complex local biholomorphic maps depend on $n$ holomorphic functions of $n$ arguments.

In the other (non-integrable) cases we will show that the number of arguments $p\le n-1$,
and we will study the structures that maximize $\z$. We will prove
(this shall be understood either in analytic category or formally):
 \begin{theorema}\label{thB}
If $N_J\ne0$, then the local transformations from the pseudogroup $G$ depend on
at most $\z$ functions of $(n-1)$ arguments, where $\z=n-1$ for $n=2,3$ and $\z=n-2$ for $n>3$.
 \end{theorema}

For the sub-maximal symmetric structures, which realize the above dimensional characteristics,
some vector distributions, associated to $J$, will be shown integrable.
Thus $M$ comes naturally equipped with pseudoholomorphic foliations.
This feature is lacking generally, namely $J$-holomorphic foliations of complex dimension $>1$
are non-existent for generic $J$ \cite{Gr,K$_2$}; they do exist if the fibers are holomorphic curves,
but generically $J$ is not locally projectible along the fibers, which is though the case with
the above canonical foliations.

In the case $n=2$ the sub-maximal symmetric structure is locally unique, and is the only nonintegrable $J$
that admits an infinite-dimensional transitive symmetry pseudogroup $G$.
It corresponds to almost holomorphic vector bundle over a holomorphic curve.
We also investigate the geometry of the canonical rank 2 Nijenhuis tensor
characteristic distribution $\Pi^2=\op{Im}(N_J)$. It is shown that locally this rank 2
vector distribution can be arbitrary, but it has to be integrable if $\dim G=\infty$.
On the contrary, globally in dimension 4 the existence of an almost complex structure $J$
with $N_J\ne0$ (so that the field of 2-planes $\Pi^2$ is non-singular)
imposes topological restrictions on $M$.

In the case $n=3$ we describe the sub-maximal symmetric structures, which are again locally unique.
We discuss the almost complex structures $J$ with non-degenerate Nijenhuis tensors $N_J$
(then $J$ is called non-degenerate too).
Again existence of non-degenerate structures is a global topological restriction on $M$.
For non-degenerate $J$ we prove that $\dim G$ at most 14 and the maximal symmetric non-degenerate
almost complex structure is locally isomorphic to either the $G_2^c$-invariant structure $J$ on $S^6$
or to its split version $G_2^n$-invariant structure $J$ on $S^{2,4}\simeq S^2\times\R^4$
(see Section \ref{S4} and appendices for details).

Thus we can refine information about $\dim G$ for $n=2,3$.
Estimate for $n=2$ was known earlier \cite{KiL}, but our proof is independent and is much shorter
(the first 3 sentences in the proof of Theorem \ref{Thm2}).
We formulate the global version of the result for the group of automorphisms.

 \begin{theorema}\label{thC}
Let the almost complex structure $J$ on $M^{2n}$ be non-degenerate in the sense of Theorem \ref{thA}.
If $n=2$, then $\dim\op{Aut}(M,J)\le\dim G\le4$ and the equality $\dim\op{Aut}(M,J)=4$ is attained
only for left-invariant almost complex structures on a Lie group $M^4$.
If $n=3$, then $\dim\op{Aut}(M,J)\le\dim G\le14$ and the equality
$\dim\op{Aut}(M,J)=14$ is attained only when $M=S^6$ and $J$ is $G_2^c$-invariant
or $M=S^{2,4}$ and $J$ is $G_2^n$-invariant.
 \end{theorema}

The analysis for $n>3$ is harder due to rapidly increasing number of cases, but
we describe the sub-maximal symmetric structures that are no longer unique.
Namely the two sub-maximal structures $J$ in dimensions 4 and 6, naturally extended to
dimension $2n$, give rise to two different sub-maximal almost complex structures
(with the same bulk of local symmetries depending upon $(n-2)$ complex functions of 2 variables).
We also discuss at the end of the paper the topological obstructions for construction of
global non-degenerate almost complex structures $J$.

The theorems announced this Introduction are cumulative results from the main text:
Theorems \ref{thA} and \ref{thB} are combinations of Theorems \ref{Thm1}, \ref{Thm2}, \ref{trm7}
and Theorem \ref{thC} is a combination of Theorems \ref{Thm2}, \ref{trm5} and Corollary \ref{Cor1}.

The zoo of all almost complex structures with infinite pseudogroup $G$ is too vast
to be handled in general, but our results show that these cases contain some integrability
features, like existence of pseudoholomorphic foliations.
Thus, similar to non-degenerate situation, we also observe topological obstructions for existence of
highly symmetric almost complex structures globally.

Summarizing, this paper makes the first steps in understanding the pseudogroup
$G$ of local symmetries (automorphisms) of almost complex structures $J$
that lie between the two poles -- integrable and generic.

\section{Background: Formal integrability}\label{S1/2}

 \abz
Let us summarize here the basics from the geometric theory of differential equations
that we need, see \cite{Sp,KL$_2$} for details. All structures used in this paper
are assumed smooth unless otherwise stated.

We consider a vector bundle $\pi:E\to M$ with fiber $F$ and the space $J^k\pi$ of $k$-jets of
its local sections. Notice that $J^0\pi=E$ and there are natural projections $\pi_{k,k-1}:J^k\pi\to J^{k-1}\pi$.

A system of differential equations of order $r$ is a subbundle $\E\subset J^r\pi$.
In this paper the order of $\E$ will always be $r=1$.
The \emph{symbol} of this equation is the family of subspaces $\g_1=\op{Ker}(d\pi_{1,0}:T\E\to TJ^0\pi)$.
It naturally identifies as the subbundle $\g_1\subset T^*M\ot F$.

The Spencer-Sternberg prolongations of the symbol $\g_{k+1}=\g_1^{(k)}$ \cite{Sp,St} are given
by the formula $\g_{k+1}=S^{k+1}T^*M\ot F\cap S^kT^*M\ot\g_1$.

Prolongations of $\E$ are defined as subsets $\E_{k+1}=\E^{(k)}\subset J^{k+1}\pi$,
which are zero loci of the differential corollaries of the PDEs defining $\E$ (obtained
by differentiating the defining relations by all variables $\le k$ times).

Regularity of $\E$ means that $\E_k$ are vector subbundles of $J^k\pi$ with respect to
the projections $\pi_{k,k-1}$ for all $k$.
$\E$ is \textit{formally integrable} (all compatibility conditions hold) if and only if
$\pi_{k,k-1}:\E_k\to\E_{k-1}$ are submersions.

The symbols of $k$-th order $g_k=\op{Ker}(d\pi_{k,k-1}:T\E_k\to T\E_{k-1})\subset S^kT^*M\ot F$
are subspaces of the Spencer-Sternberg prolongations $\g_k$,
and we have equality $g_k=\g_k$ iff $\E$ is formally integrable.

Thus purely algebraic considerations allow to bound the Hilbert polynomial
 $$
P_\E(k)=\sum_{i\le k}\dim g_i=\z k^p+\dots.
 $$
The degree $p$ is called the \emph{functional dimension} and the leading coefficient $\z$ is called
the \emph{functional rank} of the system $\E$ \cite{KL$_1$}. According to Cartan test \cite{BCG$^3$}
these quantities determine the amount of initial data needed to specify formal\footnote{We can
change this to 'local' in the analytic and sometimes also in the smooth category.} solutions of $\E$.
Namely the solutions depend on $\z$ functions of $p$ arguments (and possibly some functions with
fewer arguments).

These numbers can be calculated from the complexification $\g_1^\C$ of the symbol as follows
(we restrict to the order $r=1$, see \cite{BCG$^3$,KL$_1$} for the general case).
Consider rank one elements $\varrho\ot v\in \g_1^\C\setminus0$.
The set of all such $\varrho$ form the complex affine \emph{characteristic variety}
$\op{Char}^\C_\text{aff}(\E)\subset T^*M$, and the set of all $v$ form the \emph{kernel bundle} $K$ over it.
We can also calculate these through prolongations (higher symbols) as follows:
$\varrho^k\ot v\in \g_k^\C\setminus0$.

Now $p=\dim_\C\op{Char}^\C_\text{aff}(\E)$ and $\z=\op{rank}(K)\cdot\op{deg}\op{Char}^\C_\text{aff}(\E)$
provided the characteristic variety is irreducible.
If the characteristic variety is reducible, then $p$ is the maximal dimension among the irreducible
components, and $\z$ is the sum of the corresponding quantities for irreducible pieces of dimension $p$.

 \smallskip

{\bf Example.} Consider the Cauchy-Riemann equation for the holomorphic map $w:\C^n\to\C^m$,
$z\mapsto w(z)$. It is given as the system $\p_{\bar z_k}w^j=0$, where $2\p_{\bar z_k}=\p_{x_k}+i\p_{y_k}$.
The covector $\xi=(\xi_1'+i\xi_1'',\dots,\xi_n'+i\xi_n'')$ is characteristic
iff the $2\times 2n$ real matrix with blocks $\begin{pmatrix} \xi_k' & \xi_k'' \\ -\xi_k'' & \xi'_k\end{pmatrix}$
has rank 1 (the actual matrix of the Cauchy-Riemann operator has size $2m\times 2n$, but its rows have
repetitions). Complexification of these equations writes as
$(\xi_1'',\dots,\xi_n'')=\pm i(\xi_1',\dots,\xi_n')$.
Thus $\op{Char}^\C_\text{aff}(\E)$ is reducible and consists of 2 planes
of dimensions $p=n$, and $K$ is a bundle of rank $n$ over each of them. Thus $\z=2n$.

We conclude that the solution of the Cauchy-Riemann equation depends on $2n$ (real-valued)
functions of $n$ arguments. In fact, it depends on $n$ holomorphic functions of
$n$ complex arguments.

\smallskip

In formal calculus of dimensions developed below we can interpret paramet\-rization of the
solution space via solutions of a (formally integrable) non-linear Cauchy-Riemann equation
$\bar\p(f)=\Phi(z,f)$ (instead of holomorphic functions); the space of such $f$ is
isomorphic to the space of complex-analytic functions via the initial value problem.
But for simplicity we will just mention that the solutions depend on $\frac12\z$ complex
functions of $p$ (complex) arguments.

\section{Finiteness for almost complex automorphisms}\label{S1}

 \abz
Contrary to the global situation (when closedness, Kobayashi hyperbolicity or
other conditions guarantee the group of symmetries to be small), locally
$(M,J)$ can have an infinite-dimensional Lie algebra sheaf of symmetries.
We begin with an effective criterion to check finite-dimensionality.

For every $x\in M$ let us endow the tangent space $T=T_xM$ with the structure
of module over $\C$ by the rule
$(\a+i\b)\cdot v=\a\cdot v+\b\cdot Jv$ for $\a,\b\in\R$, $v\in T$.
The Nijenhuis tensor is then an element of the bundle $\La^2T^*\ot_{\bar\C}T$
(the wedge product is over $\C$).
The image of the Nijenhuis tensor is the $\C$-linear subspace $\op{Im}(N_J)=N_J(T\we T)\subset T$.

We call the Nijenhuis tensor \emph{non-degenerate} if its rank as a map $N_J:\La^2_\C T\to T$ is maximal.
This means $N_J\ne0$ for $n=2$ and $\op{Im}(N_J)=T$ for $n>2$.

 \begin{theorem}\label{Thm1}
{\rm(i)} If $n>2$ and $N_J$ is non-degenerate, then the the pseudogroup
$G=\op{Sym}_\text{\rm loc}(M,J)$ is finite-dimensional (and so is a Lie group).\\
{\rm(ii)}
If $N_J\ne0$, then the pseudogroup $G$ depends on at most $n-1$ complex functions
in $n-1$ variables.
 \end{theorem}


 \begin{proof}
{\rm (i)}
Let us consider the Lie equation on the 1-jets of infinitesimal symmetries $X\in\mathfrak{D}_M$
(space of vector fields on $M$) at various points $x\in M$ preserving the structure $J$:
 $$
\mathfrak{Lie}(J)=\{[X]^1_x:L_X(J)_x=0\}\subset J^1(TM).
 $$
Its symbol is $T^*\ot_\C T$ and this equation is formally integrable iff $J$ is integrable.
So we consider the first {\it prolongation-projection\/} of this equation $\E=\pi_1(\mathfrak{Lie}(J)^{(1)})$,
which is the Lie equation for the pair $(J,N_J)$ consisting of 1-jets of vector fields
preserving both tensors. The symbol of the equation $\E$ is
 $$
\g_1=\{f\in T^*\ot_\C T: N_J(f\xi,\eta)+N_J(\xi,f\eta)=fN_J(\xi,\eta)\,\forall\xi,\eta\in T\}.
 $$
The Spencer-Sternberg prolongation $\g_1^{(1)}$ of this space
equals (the symmetric tensor product $S^k$ everywhere below is over $\C$):
 $$
\g_2=\{h\in S^2T^*\ot_\C T: N_J(h(\xi,\eta),\zeta)+N_J(\eta,h(\xi,\zeta))
=h(\xi,N_J(\eta,\zeta))\}.
 $$
We can substitute $J\xi$ instead of $\xi$ in the defining relation for $h$. Then
anti-linearity of $N_J$ implies $h(W,\cdot)=0$ for $W=\op{Im}(N_J)$.
If $N_J$ is non-degenerate, this implies that $\g_2=0$.

The same holds in a more general case, when $W^\perp=\{v\in T:N_J(v,W)=0\}$
is zero. Indeed, the defining relation yields
 \begin{equation}\label{NJJ}
N_J(h(\xi,\eta),\zeta)=N_J(h(\xi,\zeta),\eta),
 \end{equation}
so taking $\zeta\in W$ we get $h(\xi,\eta)\in W^\perp$.

When $\g_2=0$, the system $\mathfrak{Lie}(J)$ has finite type
\cite{Sp,ALV,KL$_2$}. Consequently $G$ is a Lie group \cite{Ko},
and $\dim G=\dim\g_0+\dim\g_1<2(n^2+n)$.

\smallskip

{\rm (ii)} In the general case let us observe that the first part of the proof
implies that the higher prolongations satisfy ($k>1$)
 \begin{equation*}
\g_k\subset S^k\op{Ann}(W)\ot_\C W^\perp,
 \end{equation*}
where $\op{Ann}(W)=\{\varrho\in T^*:\varrho(W)=0\}$ is the annihilator.
Indeed, for $k=2$ this follows from the analysis of $\g_2$, and
the claim for $k>2$ follows by prolongation.

Thus the characteristic variety of the Lie equation $\E=\mathfrak{Lie}(J,N_J)$ (and so also of
$\mathfrak{Lie}(J)$) satisfies
 $$
\op{Char}^\C_\text{aff}(\E)\subset(\op{Ann}W)_\C^{1,0}\cup(\op{Ann}W)_\C^{0,1}\subset{}^\C T^*
 $$
and so the maximal value of $p$ is $n-1$.

To calculate the kernel space $K_\varrho$ over $\varrho\in\op{Char}^\C_\text{aff}(\E)$ consider
$h=\varrho^2\ot v\in\g_2^\C$, $v\in K_\varrho$. Setting $\eta\in\op{Ann}(\varrho)=\{u\in T:\varrho(u)=0\}$,
$\xi=\zeta\not\in\op{Ann}(\varrho)$ into (\ref{NJJ}) we obtain:
$h(\xi,\eta)=0$ $\Rightarrow$ $N_J(v,\eta)=0$. Therefore
 $$
K_\varrho=\{v\in T:N_J(v,\op{Ann}(\varrho))=0\}\subset W^\perp.
 $$
This space is $J$-invariant and is strictly smaller than $T$. Indeed since
$\op{Ann}(\varrho)$ is a hypersurface, equality $\op{Ann}(\varrho)^\perp=T$ would imply
$N_J=0$.
Therefore $\frac12\z\le(n-1)$.
 \end{proof}



For $n=2$ the image of the Nijenhuis tensor never spans $T$, and a finer criterion is
provided in the next section.

\section{Almost complex structures in dimension 4}\label{S2}

 \abz
In dimension $2n=4$ the induced $\op{GL}(2,\C)$ action on the space $\La^2T^*\ot_{\bar\C}T$ of
Nijenhuis tensors has exactly two orbits: zero and its complement.
For non-zero tensor $N_J$ its image as a map $\La^2T\to T$ is a rank 2 distribution in $M^4$ (\cite{K$_2$}).

 \begin{dfn}
$\Pi^2=\op{Im}(N_J)\subset T$ is called the Nijenhuis tensor characteristic distribution.
 \end{dfn}
Provided $J$ is generic, $\Pi^2$ is generic as well (see \S\ref{S4}). In particular, there is the derived
rank 3 distribution $\Pi^3=\p\Pi^2$ formed by the brackets of sections of $\Pi$.

This leads to the invariant e-structure $\{\xi_i\}_{i=1}^4$, $\xi_i\in\mathfrak{D}_M$, as follows:
 $$
\x_1\in C^\infty(\Pi^2),\ \x_3\in C^\infty(\Pi^3),\ N_J(\x_1,\x_3)=\x_1,\ \x_2=J\x_1,\ [\x_1,\x_2]=\x_3,\ \x_4=J\x_3.
 $$
These formulae define the pair $(\xi_1,\x_2)$ up to multiplication by $\pm1$ and the pair $(\x_3,\x_4)$ absolutely canonically \cite{K$_3$}.

 \begin{theorem}\label{Thm2}
If $\p\Pi^2$ is a rank 3 distribution, then the group $G$ of local symmetries is at most
4-dimensional. Moreover $G$ is finite-dimensional unless $\Pi$ is integrable and $J$ is
projectible along the fibers of the corresponding foliation\footnote{Denoting the (local)
projection along fibers by $\pi$, we get that $J$ is projectible if $\pi_*J=J_0\pi_*$
for some almost complex structure $J_0$ on the quotient. Similarly, $N_J$ is projectible if
$\pi_*N_J=N_0\pi_*$ for some $(2,1)$-tensor $N_0$ on the quotient.}.

If $G$ is infinite-dimensional, then its orbits contain the leaves of $\Pi$,
and $G$ is formally parametrized by 1 complex function of 1 argument.
 \end{theorem}

 \begin{proof}
It is known \cite{Ko} that the automorphism group of an e-structure on a manifold $M$
is a Lie group of dimension at most $\dim M$, which is 4 in our case. The additional $\pm$
can only change the amount of connected components.
Thus in the case of non-integrable $\Pi$ we get $\dim G\le4$.

For general distribution $\Pi$ from the proof of Theorem \ref{Thm1} we have
 $$
\g_2=\{h\in S^2T^*\ot_\C T: N_J(h(\xi,\eta),\zeta)=N_J(h(\xi,\zeta),\eta),\ h(\xi,N_J(\eta,\zeta))=0\}.
 $$
These relations imply $h(\Pi,\cdot)=0$ and $h(\xi,\eta)\in\Pi$.
Thus $\g_2=S^2\op{Ann}(\Pi)\ot_\C\Pi$ and more generally $\g_k=S^k\op{Ann}(\Pi)\ot_\C\Pi$, $k>1$.
The characteristic variety equals
 $$
\op{Char}^\C_\text{aff}(\E)=(\op{Ann}\Pi)_\C^{1,0}\cup(\op{Ann}\Pi)_\C^{0,1}\subset{}^\C T^*.
 $$
The kernel bundle over $\op{Char}^\C_\text{aff}(\E)$ is the 1-dimensional complex line $\Pi$,
and this (together with the fact that the symmetries are real) implies the result on functional dimension.

Now let us note that $\nu=T/\Pi$ is a Riemannian bundle over $M$.
The orthogonality is given by $J$ and the unit circle $S^1\subset\nu_x$ ($x\in M$ is an arbitrary point)
consists of vectors $v\in\nu_x$
satisfying the condition $N_J(v,\cdot)\in\op{SL}(\Pi_x)$.

Let $\mathcal{O}$ be the orbit of the pseudogroup $G$ action through $x$ and $V=T_x\mathcal{O}$
the tangent space, which is obtained by evaluation of all infinitesimal symmetry fields at $x$.
If $V\cap\Pi=0$, then the pseudogroup is at most 3-dimensional as it preserves the
Riemannian metric on $\mathcal{O}$.

Consider the case $\dim V\cap\Pi\ne0$, and suppose $G$ is infinite-dimensional.
Then (as the transversal symmetries are at most 3-dimensional) we can take a vector field $\xi$
tangent to $V\cap\Pi$ which coincides with a symmetry at $x$ up to the second order of smallness.
Then for a vector field $\eta$ transversal to $\Pi$ we have at $x$:
$[\xi,J\eta]=J[\xi,\eta]\,\op{mod}\Pi$, which implies $[J\xi,J\eta]=J[J\xi,\eta]\,\op{mod}\Pi$ at $x$.
Thus $L_X(J)(\eta)=0\,\op{mod}\Pi$ $\forall X\in\Pi$, $\eta\in\nu$, so that $J$ is projectible along $\Pi$.
 \end{proof}

 \begin{rk}
If $J$ is projectible along the leaves of $\Pi$, then $N_J$ is also projectible.
The inverse is not true, see an example in \cite{K$_3$}.
 \end{rk}

The theorem holds under the regularity assumptions (in a neighborhood of a regular point),
when all the involved bundles have constant ranks. It does not obviously extend to singular points, because
there are examples of linear equations with sheaf of solutions being finite-dimensional
at every regular point, yet infinite-dimensional at some singular points \cite{K$_6$}.
The claim is however true globally in analytic category.

\section{Sub-maximal symmetric structures for $n=2$}\label{S3}

 \abz
Now we describe when the Lie equation $\E=\mathfrak{Lie}(J,N_J)$ from \S\ref{S1}
is compatible (formally integrable) in the case $M$ is 4-dimensional.

In \cite{LS} {\it almost holomorphic vector bundles\/} $\pi:(E,\hat J)\to(\Sigma,J')$ were defined as such
pseudoholomorphic bundles that the fiber-wise addition $E\times_\Sigma E\to E$ is a
pseudoholomorphic map. In \cite{K$_4$} we gave a constructive version and related it
to the theory of normal pseudoholomorphic bundles.

 \begin{theorem}\label{Thm3}
Suppose $J$ is non-integrable everywhere, i.e. $N_J|_x\ne0$ $\forall x\in M$. Then
$G=\op{Sym}_\text{\rm loc}(M,J)$ is infinite-dimensional iff $J$ is (locally) isomorphic
to an almost holomorphic vector bundle structure
$(M^4,J)\to(\Sigma^2,J_0)$ over a Riemannian surface.

If $G$, in addition, acts transitively, then in local coordinates we have the normal form
 $$
J\p_z=i\,\p_z+w\,\p_{\bar w},\quad J\p_w=i\,\p_w.
 $$
 \end{theorem}


 \begin{proof}
Let $\dim G=\infty$.
From the necessary conditions of Theorem \ref{Thm2} we know that projection along the leaves
of the integrable distribution $\Pi$ gives us locally the pseudoholomorphic map
$\pi:(M,J)\to(\Sigma,J_0)$. Choosing a complex coordinate $z$ for the complex structure $J_0$
on the holomorphic curve $\Sigma$ and a complex coordinate $w$ along the fibers of $\pi$ we
get the local description of $J$
 \begin{equation}\label{intP}
J\p_z=i\,\p_z+a\,\p_w+b\,\p_{\bar w},\quad J\p_w=i\,\p_w
 \end{equation}
(action on $\p_{\bar z},\p_{\bar w}$ is obtained by conjugation due to reality of $J$;
cf. the real version of the above description in \cite{K$_4$}). The condition $J^2\p_z=-\p_z$ yields $a=0$.
Then we calculate $N_J(\p_z,\p_w)=-2i\,b_w\p_{\bar w}$, whence $\Pi=\op{Re}(\C\cdot\p_w)$ according
to our setup and $b_w\ne0$ by the assumption.

The equation for pseudoholomorphic curves in $M$ transversal to $\Pi$ writes
 \begin{equation}\label{e-1}
w_{\bar z}=\psi,\quad \text{ where }\psi=\tfrac12\,i\,\bar b,
 \end{equation}
which we couple with the conjugate into
$\E=\{w_{\bar z}=\psi,{\bar w}_z=\bar\psi\}\subset J^1(\Sigma,M)_\C$
(the latter is the complexified space of jets of $\pi$).
The almost complex structure $J$ on $M$ reads off (up to $\pm$) from
the collection of all these $J$-holomorphic curves,
so the symmetries of $J$ coincide with the (real) symmetries of the equation $\E$.

Passing to vector fields, the Lie algebra sheaf of $G$ corresponds to the space of
infinitesimal point symmetries of $\E$, which we search by the method of S.Lie,
see \cite{KL$_2$} and references therein.
Since we know that most of the symmetries ($\infty$ with the possible exception of 3)
are tangent to $\Pi$ (fibers of $\pi$), we write the point symmetry as
$\xi=f\,\p_w+\bar f\,\p_{\bar w}$ (this removal of the translational part is
standard in the Lie symmetry method). The prolongation of $\xi$ to the space of 1-jets
(the doted quantities are not essential)
 $$
\hat\xi=f\,\p_w+\bar f\,\p_{\bar w}+(\dots)\p_{w_z}+(\dots)\p_{{\bar w}_z}
+(f_{\bar z}+w_{\bar z}f_w+{\bar w}_{\bar z}f_{\bar w})\p_{w_{\bar z}}+(\dots)\p_{{\bar w}_{\bar z}}
 $$
acts as follows (the second equation of $\E$ gives the conjugate relation)
 $$
\hat\xi(w_{\bar z}-\psi)|_\E=f_{\bar z}+\psi\,f_w+{\bar w}_{\bar z}f_{\bar w}-f\psi_w-\bar f\psi_{\bar w}=0.
 $$
Since ${\bar w}_{\bar z}$ is arbitrary, this obviously decouples into
 \begin{equation}\label{e0}
f_{\bar w}=0\quad\text{ and }
 \end{equation}
 \begin{equation}\label{e1}
f_{\bar z}+\psi\,f_w=f\psi_w+\bar f\psi_{\bar w}.
 \end{equation}
We will investigate this overdetermined system $\mathcal{S}_J$ by uncovering the
compatibility conditions. Differentiating (\ref{e1}) by $\bar w$ yields
 \begin{equation}\label{e2}
f_w-\overline{f_w}=(\log\psi_{\bar w})_wf+(\log\psi_{\bar w})_{\bar w}\bar f.
 \end{equation}
Taking the sum of (\ref{e2}) and its conjugate results in
 \begin{equation}\label{e3}
(\log|\psi_{\bar w}|^2)_wf+(\log|\psi_{\bar w}|^2)_{\bar w}\bar f=0.
 \end{equation}
If this relation is nontrivial, then $f=h\cdot f_0$, where $f_0$ is a fixed complex function
and $h$ a real indeterminate. Substitution of this into $\mathcal{S}_J$ gives a finite
type system, so the symmetry algebra is finite-dimensional.

Thus we must assume that both coefficients of (\ref{e3}) vanish, so that $|\psi_{\bar w}|=k=\op{const}$
by $w$. We write $\psi_{\bar w}=ke^{i\theta}$, where $\theta$ is a real variable (mod\,$2\pi$).
Differentiation of (\ref{e2}) by $w$ and by $\bar z$ gives:
 \begin{equation}\label{e4}
f_{ww}=(\log\psi_{\bar w})_wf_w+(\log\psi_{\bar w})_{ww}f+(\log\psi_{\bar w})_{w\bar w}\bar f.
 \end{equation}
 \begin{equation}\label{e5}
(\log{\bar\psi}_w)_{w\bar w}\overline{f_z}=\Psi(f,\bar f,f_w,\bar{f}_w),
 \end{equation}
where $\Psi$ is some linear function in the indicated arguments with coefficients depending
differentially on $\psi$. 
Differential corollaries (\ref{e4})-(\ref{e5})
together with the original equations (\ref{e0})-(\ref{e1}) form a finite type system unless
 \begin{equation}\label{e6}
(\log{\bar\psi}_w)_{w\bar w}=0\quad\Leftrightarrow\quad(\log\psi_{\bar w})_{w\bar w}=0.
 \end{equation}
So we obtain another restriction on $\psi$: $\theta$ is harmonic.

Now the choice of complex coordinate $w$ in the leaves of $\Pi$ was arbitrary, and we wish to
normalize it\footnote{A general fiber-wise diffeomorphism preserves the symmetry algebra
but changes the coordinate form of the Cauchy-Riemann equation (\ref{e-1}): linear
to nonlinear etc. An alternative way to the normalization is to find all restrictions on $\psi$ and
then to check linearization of the Cauchy-Riemann equation.}.
We can suppose that the infinite pseudogroup $G$ contains a fiber-wise symmetry with no stationary
points. Then we can fix it to be translation in each leaf, i.e. the generating vector field
is $q\p_w+\bar q\p_{\bar w}$, where $q=q(z,\bar z)$. In other words
we can assume (locally) that $\mathcal{S}_J$ has a nonzero solution $f$ with $f_w=0$.
Then derivatives of (\ref{e2}) imply that the system of vectors
 $$
(\theta_w,\theta_{\bar w}),\quad  (\theta_{ww},0)\quad (0,\theta_{\bar w\bar w})
 $$
has rank $\le1$. Together with (\ref{e6}) this implies $\theta=\theta_0+\a w+\bar\a\bar w$
with real $\theta_0$ and complex $\a$ functions of $z,\bar z$.
Moreover (\ref{e1}) yields $\psi=\vp_0(z,\bar z)+\vp_1(z,\bar z)w+\Phi_{-1}$, where
$\Phi_{-1}=\vp_{-1}(z,\bar z)\bar w$ for $\alpha=0$ and $\Phi_{-1}=\frac{k}{i\bar\a}e^{i\theta}$
else.

Let us consider the case $\a=0$ first.
Substitution of the expression for $\psi$ into (\ref{e4}) and (\ref{e2}) implies that $f_w$ is real
and constant by $w$, i.e. we have: $f=f_0(z,\bar z)+f_1(z,\bar z)w$ ($f_1\in\R$).
This together with $w$-derivative of (\ref{e1}) give us $f_{1\bar z}=0$, i.e. $f_1=c$ is a real constant.

Since the first term in $\psi=\vp_0(z,\bar z)+\vp_1(z,\bar z)w+\vp_{-1}(z,\bar z)\bar w$
can be eliminated by a shift in $w$, $J$ is an almost pseudoholomorphic vector bundle structure.
The general symmetry $f=f_0(z,\bar z)+cw$
is a combination of the scaling symmetry $\op{Re}(w\p_w)$ and the shift by solution
$\op{Re}(f_0(z,\bar z)\p_w)$, where $f_0$ satisfies the linear (nonhomogeneous)
Cauchy-Riemann equation
 $$
w_{\bar z}=-c\,\vp_0(z,\bar z)+\vp_1(z,\bar z)w+\vp_{-1}(z,\bar z)\bar w.
 $$
The solutions of the latter are parametrized by 1 function in 1 argument.

Via shift by a solution and variation of constants the linear Cauchy-Riemann equation takes
the normal form
 \begin{equation}\label{wzl}
w_{\bar z}=\l(z,\bar z)\bar w,
 \end{equation}
where $\l\ne0$ corresponds to $N_J\ne0$. The gauge changes preserving this type of equation
are $z\mapsto Z(z)$, $w\mapsto U(z)w$ (functions $Z,U$ are holomorphic),
and so $\Lambda=(\log\l)_{z\bar z}\,dz\,d\bar z$ is an invariant quadratic form.

Consider now the case $\a\ne0$. Here substitution of
$\psi=\vp_0(z,\bar z)+\vp_1(z,\bar z)w+\frac{k}{i\bar\a}e^{i\theta}$
in (\ref{e2}) gives $f_w=i\a f+r$, where $r=r(z,\bar z)$ is a real function. This implies
$f=\frac{i r}{\a}+c\,e^{i\a w}$, where $c=c(z,\bar z)$ is a complex function.
Substituting both this and the expression for $\psi$ into (\ref{e1}) yields
 $$
(\tfrac{i r}{\a})_{\bar z}+(c_{\bar z}+i c\,\a_{\bar z}w)\,e^{i\a w}+i\a(\vp_0+\vp_1w)c\,e^{i\a w}=
\tfrac{i r}{\a}\vp_1+c\vp_1e^{i\a w}+k\bar c\,e^{i\theta_0+i\a w}.
 $$
Splitting this equation according to $w$-parts we get $\vp_1=-\frac{\a_{\bar z}}{\a}$,
$r_{\bar z}=0$ (so that $r$ is a real constant) and a Cauchy-Riemann type equation on $c$.

Let us change the variable $w\mapsto \z\,e^{-i\a w}$, where $\z$ satisfies
$(\log\z)_{\bar z}=i\a\vp_0$. Then direct calculation shows that the equation $w_{\bar z}=\psi$
is transformed into the equation of type (\ref{wzl}), so we can proceed as in the case $\a=0$.

Notice that in addition to $\Lambda$ there is another canonical quadric (Riemannian metric)
$Q=|\l|^4dz\,d\bar z$ on $\Sigma=M/\Pi\simeq\C(z)$, discussed in the proof of Theorem~\ref{Thm2}
($T\Sigma=\nu$). The condition that the symmetry group $G$ acts transitively implies that the space
of (local) Killing fields for this Riemannian metric is 3-dimensional, and hence it is of constant curvature.
This implies that $|\l/\l_0|^2=(1+\epsilon\,z\bar z)^{-1}$ for real
constants $\epsilon,\lambda_0$, whence $\Lambda=\frac{-\epsilon}{2|\l_0|^4}Q$.

However for $\epsilon\ne0$ the motion group of $Q$ does not act by symmetries on $J$
(direct calculation shows that 2 out of 3 Killing fields fail to do so),
and so to achieve transitivity of $G$-action we must assume $\epsilon=0$.

In this case $\Lambda=0$, i.e. we can get $\l=\op{const}$ by a gauge transformation in (\ref{wzl}).
Thus we obtain the biggest sub-maximal symmetric model\footnote{The maximal symmetric model
is obviously $\C^n$ in any dimension $n$.},
i.e. the model with maximal symmetry group among all non-integrable structures $J$:
the symmetry group $G$ is a semi-direct product of
the 3-dimensional group of motions $U(1)\ltimes\C$ of the plane $\C(z)$,
the scaling of the $w$-variable and the $\infty$-dimensional group of shifts by solutions
$w\mapsto w+f_0$ discussed above.
 \end{proof}

\section{Nijenhuis tensor characteristic distribution}\label{S4}

It was shown in \cite{K$_3$} that rank 2 distributions $\Pi^2$ in $\R^4$ of fixed type
(Engel, quasi-contact, Frobenius) are realizable as the Nijenhuis tensor characteristic distribution.
Moreover if $\Pi^2$ is analytic, it is realizable as well: $\Pi=\op{Im}(N_J)$.

In this section we remove the assumption of analyticity.

 \begin{lem}\label{L1}
An almost complex structure in $M^4$ has the following normal form in local coordinates $(z,w)$
with some smooth complex-valued functions $\a,\b$.
 $$
J\p_z=i\sqrt{1+|\a|^2}\,\p_z+\a\,\p_{\bar z}+\frac{i\a\bar\b}{1+\sqrt{1+|\a|^2}}\,\p_w+\b\,\p_{\bar w},
\quad J\p_w=i\,\p_w.
 $$
 \end{lem}

 \begin{proof}
We can foliate a neighborhood $U\subset(M^4,J)$ by $J$-holomorphic discs $U=\cup_\tau B_\tau$ (see \cite{NW,M,AL})
and choose a transversal pseudoholomorphic disk $D\subset U$. Let's introduce a complex $J$-holomorphic
coordinate $z$ on $D$ and pull-back it to $U$ using the projection along the foliation $B_\tau$. On every
leaf $B_\tau$ we introduce a complex $J$-holomorphic coordinate $w$. This local coordinate system
$(z,w)$ on $U\simeq\C^2$ is not $J$-holomorphic; the almost complex structure $J$ writes in it:
 $$
J\p_z=a\,\p_z+\a\,\p_{\bar z}+b\,\p_w+\b\,\p_{\bar w},\quad J\p_w=i\,\p_w.
 $$
The condition $J^2=-\1$ is equivalent to
 $$
(a+\bar a)\a=0,\ a^2+|\a|^2=-1,\ (a+i)b+\a\bar\b=0,\ {\bar\a}b+({\bar a}+i)\bar\b=0.
 $$
If $\a=0$ then $a=\pm i$, so we can take the general solution $a=ik$, $k\in\R$,
$k^2=|\a|^2+1$. Without loss of generality we can assume $k\ge1$ (because $a\approx i$ near $D$),
then the last equation can be omitted and we compute $b$.
 \end{proof}

 \begin{rk}
The choice $\a=0$ corresponds to the normal form (\ref{intP}) of $J$
with integrable Nijenhuis tensor characteristic distribution. Another normal form
in complex dimension $n=2$ is considered in \cite{To}.
 \end{rk}

 \begin{lem}\label{L2}
The Nijenhuis tensor characteristic distribution is $\C$-generated by the vector $v={\Xi_+}\cdot\xi$, where
 \begin{gather*}
\xi=\p_z-i\frac{1+\sqrt{1+|\a|^2}}{\bar\a}\,\p_{\bar z}+\frac{\bar\b}{\bar\a}\,\p_w+
\left(\frac{i\b}{1+\sqrt{1+|\a|^2}}\frac{\Xi_-}{\Xi_+}-2i\frac{\b_w}{\Xi_+}\right)\p_{\bar w}\\
\text{ and }\quad \Xi_\pm=\frac{\a_w\bar\a+\a\bar\a_w}{2\sqrt{1+|\a|^2}}\pm\frac{\a_w\bar\a-\a\bar\a_w}2.
 \end{gather*}
At the points where $\a=0$ we have
 $$
v=-i\a_w\,\p_{\bar z}+\bar\b\a_w\,\p_w-2i\b_w\,\p_{\bar w}.
 $$
 \end{lem}

 \begin{proof}
This is a straightforward calculation. Notice that $\Xi_+\ne0$, $\a\ne0$ is a sufficient condition
for $\Pi=\op{Im}(N_J)$ to be a non-singular distribution.
 \end{proof}

Notice that the coordinate system $(z,w)$ from Lemma \ref{L1} is very special,
for instance the absolute value of the coefficient of $\p_{\bar z}$ in $\xi$
is bigger than that of $\p_z$.

 \begin{theorem}
Locally any smooth rank 2 distribution $\Pi$ is the Nijenhuis tensor characteristic distribution:
$\Pi=\op{Im}(N_J)$.
 \end{theorem}

 \begin{proof}
A rank 2 distribution can be written in the coordinate system $(z,w)$ as
 $$
\Pi^\C=\langle\p_z-A\p_w-B\p_{\bar w},\p_{\bar z}-\bar B\p_w-\bar A\p_{\bar w}\rangle
 $$
(to obtain real $\Pi$ one has to take the real and imaginary parts of the first vector).

In the notations of Lemma \ref{L2} we have:
 \begin{multline*}
\xi\,\op{mod}\Pi=\left(A-i\frac{1+\sqrt{1+|\a|^2}}{\bar\a}\bar B+\frac{\bar\b}{\bar\a}\right)\p_w\\
+\left(B-i\frac{1+\sqrt{1+|\a|^2}}{\bar\a}\bar A+
\frac{i\b}{1+\sqrt{1+|\a|^2}}\frac{\Xi_-}{\Xi_+}-2i\frac{\b_w}{\Xi_+}\right)\p_{\bar w}
 \end{multline*}
Vanishing of the $\p_w$-coefficient yields
 \begin{equation}\label{DN01}
\b=-\a\bar A-i(1+\sqrt{1+|\a|^2})B
 \end{equation}
(this is also obtainable from the condition of invariance of $\Pi$ with respect to the
almost complex structure $J$ from Lemma \ref{L1}).

Substituting this into the $\p_{\bar w}$-coefficient of $\xi\,\op{mod}\Pi$ we observe
that all derivatives of $\a$ cancel and the resulting relation is
 $$
i\a\bar A_w-(1+\sqrt{1+|\a|^2})B_w=0.
 $$
This is easily solvable provided $|A_{\bar w}|>|B_w|$:
 \begin{equation}\label{DN02}
\a=\frac{-2iA_{\bar w}B_w}{|A_{\bar w}|^2-|B_w|^2}.
 \end{equation}

Now we can construct $J$. The coordinate system $(z,w)$ is at our hands. We can choose it so that
the above inequality is satisfied (even integrable $\Pi$ can be written with non-constant $A,B$).
We calculate using (\ref{DN02})
 $$
\Xi_+=\frac{4A_{\bar w}\overline{B_w}(\bar A_wB_{ww}-\bar A_{ww}B_w)}{(|A_{\bar w}|^2-|B_w|^2)^2}.
 $$
Thus the condition $\Xi_+\ne0$ translates into the inequalities $|A_{\bar w}|>|B_w|>0$ and
$\bar A_wB_{ww}\ne\bar A_{ww}B_w$. We adapt these into the choice of $(z,w)$.

Then we define the complex-valued functions $\a,\b$ by (\ref{DN02})-(\ref{DN01}) and
the almost complex structure $J$ from Lemma \ref{L1} realizes $\Pi$ as
the Nijenhuis tensor characteristic distribution.
 \end{proof}

 \begin{rk}
It is surprising that the construction is purely algebraic. This is due to the choice of special coordinates.
In \cite{K$_3$} the problem was solved via a PDE system in Cauchy-Kovalevskaya form,
that is why analyticity was assumed.
 \end{rk}

\section{Global obstructions for non-degeneracy}\label{S4.5}
Non-degeneracy of the Nijenhuis tensor is a generic condition locally, but globally
on a closed 4-manifold it is a topological constraint.

 \begin{dfn}
An almost complex structure $J$ on a 4-dimensional manifold $M$ is called
{\em non-degenerate\/} if $N_J\ne0$ everywhere.
 \end{dfn}

There are obstructions to existence of non-degenerate $N_J$, some
depending on the homotopy class of the almost complex structure $J$ and some purely topological.
Indeed, the Nijenhuis tensor determines an anti-linear isomorphism of the complex vector bundles
 $$
N_J:\La^2T\to\Pi\subset T.
 $$
This implies $c_1(\Pi)=-c_1(\La^2T)=-c_1(M)$. Denoting by $\Pi^\perp$ the normal
vector bundle equipped with the natural complex structure we have:
$c_1(M)=c_1(\Pi)+c_1(\Pi^\perp)$, $c_2(M)=c_1(\Pi)c_1(\Pi^\perp)$, whence
 $$
c_2(M)+2c_1(M)^2=0.
 $$
Since the Euler characteristic is $\chi=c_2(M)[M]$, the signature is
$\tau=\frac13p_1(M)[M]$, with $p_1(M)=c_1^2(M)-2c_2(M)$ (implying the classical Wu's
criterion for almost complex structures in dimension four:
$2\chi+3\tau=c_1^2$; see \cite{De}), and $c_1(M)^2=K\!\cdot\!K$ is the self-intersection
of the canonical class $K$ of $M$, we get the main topological restrictions
for existence of non-degenerate almost complex structure on $M$:
 $$
5\chi+6\tau=0,\quad \chi=-2\,K\!\!\cdot\!\!K.
 $$

 \begin{rk}
Reference \cite{HH} relates the existence of an almost complex structure on $M^4$ to
the existence of a rank 2 distribution. We see that in the case these two are compatible
(through $N_J$) we get stronger restrictions on $M$.
 \end{rk}

In addition there are divisibility constraints \cite{H}. Indeed by Max Noether theorem
$\chi+\tau\equiv0\,\op{mod}4$ (equivalently $c_1^2+c_2\equiv0\,\op{mod}12$),
so we must have
 $$
\chi\equiv0\,\op{mod}24.
 $$

These restrictions applied to a simply connected closed 4-manifold yield
$b_-=10+11b_+$ and $b_+\equiv1\,\op{mod}2$,
where $b_+$ and $b_-$ are the dimensions of the maximal positive and negative definite subspaces of $H^2(M)$.
For instance (the same for any type I manifold), if
 $$
\#_r\C P^2\#_s\overline{\C P^2}
 $$
possesses a non-degenerate almost complex structure, then $r=2k+1$, $s=22k+21$, $k\in\Z$
(notice that this manifold possesses an almost complex structure iff $r$ is odd).
In particular, the projective plane  blown-up at $\le20$ points
does not possess a non-degenerate almost complex structure.

For type II manifolds with the intersection form $m\begin{bmatrix}0 & 1\\ 1 & 0\end{bmatrix}+n\,E_8$,
the above relations imply $n=\frac54(m+1)>0$, $m\equiv3\,\op{mod}4$.

In particular, there exists no non-degenerate almost complex structures on K3 surfaces (then
$5\chi+6\tau=24$), on Enriques surfaces ($\chi=12$ not divisible by 24) and
on complex surfaces of Kodaira dimension 2 (surfaces of general type, then both $\chi$ and $K\!\!\cdot\!\!K>0$).

 \begin{rk}
In other words, for any almost complex structure $J$ on any of these manifolds there will exist points
$x$ with $N_J|_x=0$.
 \end{rk}

There are however closed 4-manifolds with non-degenerate almost complex structures, for instance
Abelian surfaces. Indeed on the torus $T^4=\C^2(z,w)/\Z^4$ with almost complex structure given by
 $$
J\p_z=i\,\p_z+e^{2\pi i\cdot\op{Re}(w)}\bar\p_w,\ \ J\p_w=i\,\p_w
 $$
the Nijenhuis tensor nowhere vanishes.

\section{Almost complex structures in dimension 6}\label{S5}

For $n=3$ non-degeneracy of the Nijenhuis tensor means $\C$-antilinear isomorphism
 \begin{equation}\label{Eq5a}
N_J: \La^2T\to T.
 \end{equation}
Nijenhuis tensors were classified in \cite{K$_2$}. There are 4 non-degenerate
types\footnote{The third type is obtained from the equation in \cite{K$_2$} by changing the basis
$(X_2,X_3)\mapsto(X_2\pm JX_3)$ and complex scaling.}
 \begin{enumerate}
\item[NDG(1)\hspace{-25pt}]\hspace{25pt} $N(X_1,X_2)=X_2$, $N(X_1,X_3)=\l X_3$, $N(X_2,X_3)=e^{i\vp}X_1$,
\item[NDG(2)\hspace{-25pt}]\hspace{25pt} $N(X_1,X_2)=X_2$, $N(X_1,X_3)=X_3+X_2$, $N(X_2,X_3)=e^{i\vp}X_1$,
\item[NDG(3)\hspace{-25pt}]\hspace{25pt} $N(X_1,X_2)=e^{-i\psi}X_3$, $N(X_1,X_3)=-e^{i\psi}X_2$, $N(X_2,X_3)=e^{i\vp}X_1$,
\item[NDG(4)\hspace{-25pt}]\hspace{25pt} $N(X_1,X_2)=X_1$, $N(X_1,X_3)=X_2$, $N(X_2,X_3)=X_2+X_3$,
 \end{enumerate}
Parameters $\l,\vp,\psi$ are real. For non-exceptional values of the parameters ($\l\ne\pm1$,
$-\l\ne e^{\pm2i\vp}$, $e^{\pm2i\vp}\ne\pm1$, $e^{\pm2i\psi}\ne\pm1$ and $\psi\pm\vp\not\in\pi\Z$)
there is precisely one fixed point $\Pi$ of the
composition of maps $\Phi_2:\op{Gr}^\C_2(3)\to\C P^2$, $\Pi\mapsto N_J(\La^2\Pi)$ and $\Phi_1:\C P^2\to\op{Gr}^\C_2(3)$,
$L\mapsto\op{Im}(N_J(L,\cdot))$ that satisfies $N_J(\La^2\Pi)\cap\Pi=0$ in cases NDG(1-2), three such $\Pi$
in case NDG(3) and exactly one $\Pi$ with $N_J(\La^2\Pi)\subset\Pi$ in case NDG(4).

Isomorphism (\ref{Eq5a}) imposes global obstructions for existence of a non-dege\-ne\-rate
almost complex structure $J$.
Indeed, the total Chern class $c(T)$ is related to
$c(\La^2T)=1+(2c_1(T))+(c_2(T)+c_1(T)^2)+(-c_3(T)+c_1(T)c_2(T))$
through (\ref{Eq5a}), so (these formulae were also obtained in \cite{B})
 \begin{equation}\label{Eq5b}
3c_1(J)=0,\ c_1(J)^2=0,\ c_1(J)c_2(J)=0.
 \end{equation}
We write $c_i(J)$ instead of $c_i(T)=c_i(M)$ to stress that, contrary to 4-dimensional situation,
the Chern numbers, like $c_1c_2$, obtained through these characteristic classes
depend on the homotopy class of $J$ (and not only on the diffeomorphism type of $M$).
The obstructions for existence of almost complex structure in dimension 6 \cite{H,Ge} are
$c_1c_2\equiv0\,\op{mod}24$, $c_1^3\equiv c_3\equiv0\,\op{mod}2$, and only the last requirement
$\chi(M)\in2\Z$ does not follow from (\ref{Eq5b}).

For example, there is no non-degenerate $J$ in the homotopy class of the standard complex structure of
 $$
\C P^3\#\overline{\C P^3}\#\dots\#\overline{\C P^3}.
 $$

The canonical $G_2$-invariant almost complex structure  $J$ on $S^6$ is non-dege\-ne\-rate and it
corresponds to NDG(3) $\vp=\psi=0$. But if we blow up $S^6$, then the resulting
$S^6\#\overline{\C P^3}\simeq {\C P^3}$ (orientation forgetting diffeomorphism)
has no non-degenerate $J$ in the respective homotopy class.

More generally, according to \cite{Th} the almost complex structures on $\C P^3$ are bijective with
the total Chern classes
 $$
1+2rx+2(r^2-1)x^2+4x^3,
 $$
where $x$ is the standard generator of $H^2(\C P^3)$ and $r\in\Z$
($r=2$ corresponds to the standard complex structure and $r=-1$ corresponds to
the blown-up $S^6$). There are no non-degenerate $J$ unless $r=0$ (where we don't know).

If we impose a stronger requirement of $N_J$ being non-exceptional of type NDG(1-2), then
$N_J:\La^2\Pi\simeq L$, $T=L\oplus\Pi$ imply $c_1(J)=0$, i.e. the canonical class vanishes $K=0$.
Moreover, $c_3(J)=\l(c_2(J)+\l^2)$ for an integer cohomology class $\l\in H^2(M)$ ($\l=-c_1(\Pi)$).

For non-exceptional NDG(3) we also have $K=0$ and in addition $-c_2(J)=\l^2+\l\mu+\mu^2$, $-c_3(J)=\l\mu(\l+\mu)$
for some integer classes $\l,\mu\in H^2(M)$ (these conditions imply the above condition for the type
NDG(1-2), but not otherwise).

If $N_J$ is of type NDG(4), then $c_2(J)=\l^2+c_1(J)\l$, $c_3(J)=-2\l^3$ for some integer
cohomology class $\l\in H^2(M)$ ($\l=-c_1(\Pi)$).

For instance $\C P^3$ contains neither non-exceptional type NDG(3) nor non-exceptional type NDG(4) structures;
the sphere $S^6$ has no non-exceptional almost complex structure (of any type).

\section{Nondegenerate symmetric structures for $n=3$}\label{S6}

By Theorem \ref{Thm1} an almost complex structure on a 6-dimensional manifold has a finite-dimensional
symmetry algebra provided $N_J$ is non-degenerate.

An alternative approach to prove this is via the Tanaka theory \cite{Ta}.
Namely, the complexified tangent bundle ${}^\C T$ has the canonical subbundle $V=T^{1,0}$
with $\bar V=T^{0,1}$. Non-degeneracy of $N_J$ is equivalent to step-2 bracket-generating property:
$[V,V]={}^\C T$. Thus the growth of the distribution is $(3,6)$ and consequently the
symmetry group is finite-dimensional\footnote{The standard theory concerns real distributions in
the tangent bundle but can be generalized to real distributions in the complexified tangent bundle.
This is however a formal calculus. For real symmetries the Lie algebra of the compact $G_2$
(that act on $S^6$) lies in $so(7)$.} by \cite{Ta,K$_6$}.

The Tanaka prolongation of the corresponding graded nilpotent Lie algebra $\mathfrak{m}=g_{-2}\oplus g_{-1}$
(with $g_{-1}=V$ and $g_{-2}={}^\C T/V$) is equal to
 $$
{g}_{-2}\oplus{g}_{-1}\oplus{g}_0\oplus{g}_1\oplus{g}_2=
\La^2V\oplus V\oplus\op{gl}(V)\oplus V^*\oplus\La^2 V^*.
 $$
This is the 21-dimensional Lie algebra $so(3,4)$, but it is not the most symmetric case for non-degenerate
almost complex structures. Indeed, we have not used the natural isomorphism $N_J:\La^2V\simeq\bar V$,
which reduces the algebra of derivations of $\mathfrak{m}$ to at most 14 dimensions:

 \begin{theorem}\label{trm5}
For non-degenerate $J$ the symmetry group satisfies $\dim G\le 14$. The equality is attained
only for the exceptional Lie group $G_2$ in one of its two real versions - compact $G_2^c$ and
split (normal) $G_2^n$.
 \end{theorem}

Note that here we treat only the germ of the structure $J$, so the conclusion concerns the local group
(equivalently, its Lie algebra).

 \smallskip

 \begin{proof}
$\g_2=0$ for a non-degenerate $J$ and so $\dim G\le\dim\mathcal{O}+\dim\g_1$, where $\mathcal{O}$
is the orbit of $G$ of dimension $\le\dim M=6$. We claim that $\dim\g_1\le 8$.

This is based on the case by case analysis of the normal forms from \S\ref{S5}. Consider for instance NDG(1).
For non-exceptional parameters there exists a unique $\C$-line $L=\langle X_1,JX_1\rangle\subset T$ that is
invariant with respect to $\Phi_2\Phi_1$ and satisfies $T=L\oplus\Pi$, $\Pi=\Phi_1(L)$.
An endomorphism $f\in\g_1$ preserves this splitting and is an arbitrary complex map on $L$,
while it is given by at most 4 real parameters on $\Pi$. Thus in this case $\dim\g_1\le 6$.

In some exceptional cases the splitting is also invariant with $L=\langle X_1\rangle_\C$
distinguished by the condition that $\Phi_1(L)=\langle X_2,X_3\rangle_\C$ is generated by fixed points
$\tilde L$ of $\Phi_2\Phi_1$ satisfying $\Phi_1(\tilde L)\supset\tilde L$. Then the conclusion is the same.

But for $e^{2i\vp}=-\l=\pm1$ all complex lines $L\in\C P^2$ are fixed points of $\Phi_2\Phi_1$
(those satisfying the condition $\Phi_1(L)\supset L$ form a nondegenerate quadratic cone).
Here the group $\g_1$ acts transitively on $T$ by complex transformations preserving
 \begin{equation}\label{NeqS1}
N(X_1,X_2)=X_2,\ N(X_1,X_3)=-X_3,\ N(X_2,X_3)=X_1.
 \end{equation}
This tensor is isomorphic to (\ref{NeqS2}) and so will be discussed below.

Similarly, in the cases NDG(2), NDG(4) we get $\dim\g_1\le2$, because $\langle X_1,X_2\rangle$ is $G$-invariant etc.
For NDG(3) we have $\dim\g_1\le8$, and the equality is strict except for two cases:
 \begin{equation}\label{NeqS2}
N(X_1,X_2)=X_3,\ N(X_1,X_3)=X_2,\ N(X_2,X_3)=X_1
 \end{equation}
and
 \begin{equation}\label{NeqS3}
N(X_1,X_2)=X_3,\ N(X_3,X_1)=X_2,\ N(X_2,X_3)=X_1.
 \end{equation}
This last case is the (anti-)complexification of the usual vector product in $\R^3$,
and it is realized by the $G_2^c$-invariant almost complex structure $J$ on $S^6$.

Case (\ref{NeqS2}), equivalent to (\ref{NeqS1}), is another candidate for $\dim G=14$, and
it is realizable precisely by the $G_2^n$-invariant almost complex structure $J$ on $S^{2,4}$.
The proof uses completely different technique (representation theory)
and so is delegated to the Appendices.

It follows that a non-degenerate almost complex structure $J$ on $M^6$ has dimension of the
automorphism group at most 14. If the automorphism group has maximal dimension 14,
then its acts transitively on $M^6$ ($\dim\g_0=6$), and the stabilizer $\g_1$
is (locally) either $SU(3)$ or $SU(1,2)$. This gives the two most symmetric non-degenerate
almost complex structures in dimension 6, they live on homogeneous manifolds
$G_2^c/SU(3)$ and $G_2^n/SU(2,1)$, see the Appendices for more details.
 \end{proof}

The almost complex structures that shown up at the end of the proof are well-known
\cite{E,CG,Ka}. Let us recall their construction.

In the compact case, consider the imaginary octonions $\R^7=\op{Im}(\mathbb{O})$ equipped
with the vector product $x\times y=\op{Im}(x\cdot y)$ and the Riemannian metric
$\langle x,y\rangle=\op{Re}(-x\cdot y)$. The almost complex structure $J$ on $S^6=\{x\in\R^7:\|x\|=1\}$
given by $Jv=x\times v$, $v\in T_xS^6=x^\perp$ has the symmetry group $G_2^c$.
The Nijenhuis tensor $N_J$ is non-degenerate and has type (\ref{NeqS3}).

In the split case, consider the imaginary split-octonions $\R^{3,4}=\op{Im}(\tilde{\mathbb{O}})$
equip\-ped with the vector product $x\times y=\op{Im}(x\cdot y)$ and the pseudo-Riemannian metric
$\langle x,y\rangle=\op{Re}(-x\cdot y)$ of signature $(3,4)$. The almost complex structure $J$ on $S^{2,4}=\{x\in\R^{3,4}:\|x\|=1\}$ given by $Jv=x\times v$, $v\in T_xS^{2,4}=x^\perp$
has the symmetry group $G_2^n$. The Nijenhuis tensor is non-degenerate and has type (\ref{NeqS2}).

 \begin{cor}\label{Cor1}
The octonionic structure $J$ on $S^6$ is the most symmetric
non-degenerate almost complex structure in dimension 6 with the associated
pseudo-Hermitian structure of signature (6,0).
The split-octonionic structure $J$ on $S^{2,4}$ is the most symmetric
non-degenerate almost complex structure in dimension 6 with the associated
pseudo-Hermitian structure of signature (2,4) (or (4,2)).
 \end{cor}

The local statement (that the structure is locally homogeneous of the indicated type)
is contained in Theorem \ref{trm5}. The global version follows because the compact group $G_2^c$
is unique and has the trivial center. For the non-compact group of type $G_2$ there are two versions:
one with trivial center is what we denote $G_2^n\subset SO(3,4)$, the other
is its universal cover $\tilde{G}_2^n$, but it is not algebraic and acts through the quotient by
the center $\tilde{G}_2^n/\Z_2=G_2^n$. This explains why only two models (and no finite quotients) occur
in Theorem \ref{thC}.

Notice also that in general the structures $\pm J$ on $M^6$ are not locally equivalent, but for
the two most symmetric structures from Corollary \ref{Cor1} the antipodal map $x\mapsto-x$ maps $J\mapsto-J$.
This finishes the proof of uniqueness.

\section{Sub-maximal symmetric structures for $n=3$}\label{S7}

Now consider degenerate tensors $N_J$. By Theorem \ref{Thm1} in the case $\dim_\C\op{Im}(N_J)=2$
(denoted DG$_1$ in \cite{K$_2$}) the symmetry group $G$ depends on $\le2$ complex functions of 1
(complex) argument.

More symmetries (complex functions of 2 argument)
can occur in the case $\dim_\C\op{Im}(N_J)=1$, which has two normal forms:
 \begin{enumerate}
\item[DG$_2$(1)\hspace{-22pt}]\hspace{22pt} $N(X_1,X_2)=X_1$, $N(X_1,X_3)=0$, $N(X_2,X_3)=0$,
\item[DG$_2$(2)\hspace{-22pt}]\hspace{22pt} $N(X_1,X_2)=X_3$, $N(X_1,X_3)=0$, $N(X_2,X_3)=0$.
 \end{enumerate}
In the first case $\op{Char}^\C_\text{aff}=
(\op{Ann}\langle X_1\rangle)_\C^{1,0}\cup(\op{Ann}\langle X_1\rangle)_\C^{0,1}$.
Indeed for any $h\in\g_2$ we have $h(X_1,\cdot)=0$. Moreover the condition $N_J(h(\xi,\eta),\theta)=N_J(h(\xi,\theta),\eta)$ gives $h(\cdot,\cdot)\in\langle X_1,X_3\rangle_\C$
(via substitution $\eta=X_1$) and $h(X_3,\cdot)\in\langle X_3\rangle_\C$ (via substitution $\eta=X_3$).
There are no more relations, and we get
$\g_2=S^2\op{Ann}(\langle X_1\rangle_\C)\ot\langle X_3\rangle_\C+
S^2\op{Ann}(\langle X_1,X_3\rangle_\C)\ot\langle X_1,X_3\rangle_\C$.

Now we can calculate the characteristic variety through $\g_2$ as discussed in \S\ref{S1/2}.
The kernel bundle $K$ over $\op{Char}^\C_\text{aff}$ has complex dimension 1 on a Zariski open set
$\op{Ann}(\langle X_1\rangle_\C)\setminus\op{Ann}(\langle X_1,X_3\rangle_\C)$,
and so formally the general symmetry depends on at most 1 complex function of 2 arguments.

In order for this to be realized the two $G$-invariant distributions $\op{Im}(N_J)=\langle X_1\rangle_\C$
and $\op{Ker}(N_J)=\langle X_3\rangle_\C$ have to be integrable (otherwise their curvature tensor is another
defining relation in $G$, see \S\ref{S3}), and also their sum $\langle X_1,X_3\rangle_\C$ shall be
a holomorphic foliation of $\dim_\C=2$.
Furthermore the biggest (in the sense of dimension theory) symmetry group $G$ corresponds
to pseudoholomorphic projection along the foliation $\langle X_3\rangle_\C$ onto a 4-manifold,
which has to be the sub-maximal symmetric model $M^4$ for $n=2$.

Finally the most symmetric model of type DG$_2$(1) is $M^6=M^4(z,w)\times\C(\zeta)$.
Provided $G$ acts transitively, the almost complex structure $J$
in local coordinates $z,w,\zeta$ (corresponding to the numeration $X_2,X_1,X_3$ above)
is given by
 \begin{equation}\label{onfor}
J\p_z=i\p_z+w\p_{\bar w},\quad J\p_w=i\p_w,\quad J\p_\zeta=i\p_\zeta.
 \end{equation}
Notice that there is a 4-dimensional $J$-holomorphic foliation $\C^2(z,w)$ with non-integrable
restriction of $J$, and also lots of transversal 4-dimensional $J$-holo\-mor\-phic foliations
with integrable $J$ of the type $\Sigma^2\times\C(\zeta)$,
where $\Sigma\subset\C^2(z,w)$ is a pseudoholomorphic curve.

The bulk of symmetries form the holomorphic transformations $\zeta\mapsto\Xi(z,\zeta)$
(but there are some other symmetries depending on a function of 1 argument).

\smallskip

The other case DG$_2$(2) is more symmetric:
 \begin{theorem}\label{Thm5}
Suppose $J$ is non-integrable everywhere, i.e. $N_J|_x\ne0$ $\forall x\in M$.
Then $G$ has the biggest functional dimension iff the almost complex structure is given
by the following normal form:
 \begin{equation}\label{nofor}
J\p_z=i\p_z+\zeta\p_{\bar w},\quad J\p_\zeta=i\p_\zeta,\quad J\p_w=i\p_w.
 \end{equation}
 \end{theorem}

We shall see that the symmetries depend on 2 complex functions of 2 arguments,
but there are also functions with fewer arguments arising via Cartan's test.
The structure of the theorem maximizes the amount of this initial value data
(otherwise, as in Theorem~\ref{Thm3}, there are symmetric models with slightly
smaller pseudogroup $G$; they correspond to change $J\p_z=i\p_z+\l\p_{\bar w}$ in (\ref{nofor})).

 \begin{proof}
The characteristic variety of DG$_2$(2) is
 $$
\op{Char}^\C_\text{aff}=
(\op{Ann}\langle X_3\rangle_\C)_\C^{1,0}\cup(\op{Ann}\langle X_3\rangle_\C)_\C^{0,1},
 $$
since $h(X_3,\cdot)=0$ for any $h\in\g_2$. Moreover $\op{rank}(K)=2$.

To see this we can use the defining relation for $\g_1$ from \S\ref{S1}.
If $f=\varrho\ot v\in\g_1^\C$ with $\varrho\in\op{Char}^\C_\text{aff}$, $v\in K_\varrho$, then
$\overline{\varrho(\xi)}N_J(v,\eta)+\overline{\varrho(\eta)}N_J(\xi,v)=v\cdot\varrho(N_J(\xi,\eta))$.
Substituting $\eta=X_3$ we get $v\in\langle X_3\rangle_\C$ or $\varrho(X_3)=0$, but the first possibility
implies the second. Thus, denoting by $\theta_i$ the coframe dual to the frame $X_i$, we get
$\varrho=\a_1\theta_1+\a_2\theta_2\ne0$ $\Rightarrow$ $K_\varrho=\langle \a_2X_1-\a_1X_2,X_3\rangle_\C$.

We can also see this by calculating $\g_2$ via (\ref{NJJ}) or as the prolongation of $\g_1$.
Denoting $\Pi=\langle X_1,X_2\rangle_\C$,
$\Pi^*=\langle \theta_1,\theta_2\rangle_\C=\op{Ann}(\langle X_3\rangle_\C)$ and
$\op{sl}_2(\C)\subset\Pi^*\ot\Pi$, we get $\g_1=\op{sl}_2(\C)\oplus \Pi^*\ot\langle X_3\rangle_\C$
(tensor products over $\C$).

Thus we obtain the biggest possible functional dimension of $G$ for non-integrable $J$,
so if it is realized the case will be sub-maximal.

Compatibility conditions include integrability of the invariant distribution $\op{Im}(N_J)=\op{Ker}(N_J)=\langle X_3\rangle_\C$.
This yields a $J$-holomorphic projection of $(M^6,J)$ onto $\C^2$ with (integrable) complex structure
and it also implies constancy of the Nijenhuis tensor (otherwise the pseudogroup $G$ does not act transitively).
The arguments similar to those used in Theorem \ref{Thm3} show that in order to keep the size of $G$ the structure $J$
must come from an almost holomorphic vector bundle, and then we justify the normal form of the theorem.

To calculate the symmetries of the model
let us introduce complex coordinates $(z,\zeta,w)$ according to the order
$(X_1,X_2,X_3)$ in which $J$ takes the form (\ref{nofor}). From the above description
it is clear that $J$-holomorphic change of coordinates has the form $z\mapsto Z(z,\zeta)$,
$\zeta\mapsto\Xi(z,\zeta)$, $w\mapsto W(w,\bar w,z,\bar z,\zeta,\bar\zeta)$

The condition that $J$ has the same normal form in the new coordinate system writes as the system of PDEs
 $$
W_{\bar w}=0,\quad {\overline W}_{\zeta}=-\tfrac{i}2 Z_\zeta\,\Xi,\quad
{\overline W}_z=-\tfrac{i}2(Z_z\,\Xi-\zeta{\overline W}_{\bar w}).
 $$
This system is not yet formally integrable, and we need to add the compatibility conditions.
The resulting involutive system is ($c$ is a constant):
 \begin{gather*}
Z_z\Xi_\zeta-Z_\zeta\Xi_z=c,\quad W_w=\bar c,\quad W_{\bar w}=0,\quad
W_{\bar z}=\tfrac{i}2(\overline{Z_z}\,\overline \Xi-\bar c\,\bar\zeta),\quad
W_{\bar\zeta}=\tfrac{i}2\overline{Z_\zeta}\,\overline \Xi.
 \end{gather*}
Thus the form $\Omega=(Z_z\,\Xi-c\,\zeta)\,dz+Z_\zeta\,\Xi\,d\zeta$ is closed and
we can integrate it $\Omega=dW_0$, where $W_0=W_0(z,\zeta)$.
This implies $W=\bar c\,w+\frac{i}2\overline{W_0}+\Psi$,
where $\Psi=\Psi(z,\zeta)$ is an arbitrary holomorphic function.

Finally, the change in $(z,\zeta)$ is a biholomorphism with constant holomorphic Jacobian,
i.e. it involves 1 complex function of 2 arguments;
the other such function $\Psi$ comes from the change in $w$
(also a function with $1$ argument and a constant are involved into the general symmetry).
 \end{proof}

\section{Dimensions $2n=8$ and higher}\label{S8}

 \abz
In the case $\dim_\C M=n>3$ the orbit space of $\op{GL}(n,\C)$-action on $\La^2T^*\ot_{\bar\C}T$ is quite complicated
(as well as its real analog - normal forms of (2,1)-tensors). However we can give a sharper estimate
for the dimensional bound of $G$.

 \begin{theorem}\label{trm7}
For $n\ge4$ in the case $N_J\ne0$, the pseudogroup $G$ depends on at most $(n-2)$
complex functions of $(n-1)$ arguments.
 \end{theorem}

 \begin{proof}
By Theorem \ref{Thm1} $p\le\dim_\C\op{Ann}(W)$ for $W=\op{Im}(N_J)$ and so
the functional dimension of $G$ is $(n-1)$ iff $\dim_\C W=1$.
In this case $N_J$ can be identified with a $\C$-valued 2-form $\oo$. Let $Z\subset T$
be the kernel of $\oo$. Denote by $2m=\op{rank}_\C(\oo)$ (rank is even by Cartan's lemma).
Then $\dim_\C Z=n-2m$.

We have $\op{Ann}(\varrho)\supset W$ for a characteristic covector
$\varrho\in\op{Char}^\C_\text{aff}$
and the restriction $\oo|_{\op{Ann}(\varrho)}$ can have $\C$-rank $2m$ or $2(m-1)$
depending on whether $\op{Ann}(\varrho)$ is transversal to $Z$ or not.
The first situation is generic with respect to the choice of $\varrho$ unless $n=2m$ or $W=Z$.

Therefore the kernel space $K_\varrho=\op{Ann}(\varrho)^\perp$ (skew-orthogonality
with respect to $\oo$) is $Z$ or $Z\oplus L^1$ respective the value of the rank
$2m$ or $2(m-1)$ (in the latter case $L^1\in\op{Ker}(\oo|_{\op{Ann}\varrho})$ is
$\C$-generated by one vector not belonging to $Z$).

In other words, the functional $\C$-rank $\frac12\z$ is equal to $n-2m\le n-2$
(achieved for $m=1$) or $n-2m+1$ (achieved either for $n=2m$, when $\frac12\z=1$, or
for $W=Z$ being 1-dimensional, when $\frac12\z=2$; in any case for $n\ge4$ this branch gives
the value $<n-2$).
 \end{proof}

We see from the proof that the symmetry pseudogroup $G$ is largest in the case $N_J\ne0$,
$n\ge4$ only if $m=\op{rank}_\C N_J=1$, i.e. the only non-trivial relation for the Nijenhuis tensor
can be either $N_J(X_1,X_2)=X_1$ ($W\cap Z=0$) or $N_J(X_1,X_2)=X_3$ ($W\subset Z$).
As before we must impose integrability assumptions on the involved
distributions in order to realize this $G$.

Thus for $n\ge4$ there are 2 different cases with the largest symmetry:
the 6-dimensional manifolds described by (\ref{onfor}) or (\ref{nofor}),
multiplied by $\C^{n-3}$. In other words, the sub-maximal symmetric
models are either $M^{2n}_1=M^4\times\C^{n-2}$, where $(M^4,J)$ is the submaximal
model from Theorem \ref{Thm3}, or $M^{2n}_2=M^6\times\C^{n-3}$, where
$(M^6,J)$ is the submaximal model from Theorem \ref{Thm5}.

Notice that both cases are described by almost holomorphic vector bundle
$\pi_1:M^{2n}\to\C^{n-1}$ (with $J$ being minimal in the sense of \cite{GS}),
and that both models contain a canonical pseudoholomorphic foliation by
the kernel of the Nijenhuis tensor $\pi_2:M^{2n}\to\C^2$.

\smallskip

The other structures $J$ with large symmetry pseudogroup $G$ must also have the involved
invariant distributions integrable, so that these structures give rise to canonical
pseudoholomorphic foliations. The hierarchy of intermediate size groups $G$ is immense,
but also the totality of almost complex structures $J$ with non-degenerate $N_J$ is vast.
To study the latter the following idea is useful.

\smallskip

Consider the $\C$-antilinear map $N_J:\La^2_\C T\to T$, $T\simeq\C^n$. For non-degenerate $N_J$
the pre-image $N_J^{-1}(0)$ has complex dimension $d=\frac12n^2-\frac32n$.
The Pl\"ucker embedding $\rho:\op{Gr}^\C(2,n)\hookrightarrow P^\C\La^2T$ has image of codimension $d+3-n$.
Therefore generically $\Sigma=\varpi(N_J^{-1}(0))\cap\op{Im}\rho$ has dimension $n-4$,
where $\varpi:\La^2\to P^\C\La^2$ is the projectivization map.
The degree of subvariety $\Sigma$ is the same as the degree of a Plucker embedding of the Grassmannian
$\op{Gr}^\C(2,n)$, namely the Catalan number $\frac1{n-1}\binom{2n-4}{n-2}$.

Let us consider in more details the case $n=4$, when $\Sigma$ is zero-dimensional of degree 2.
Thus $\Sigma$ consists of two 4-dimensional $J$-invariant subspaces of $T$, which we assume
transversal $T=\Pi_1\oplus\Pi_2$ ($\Pi_i$ are characterized uniquely
by the condition $N_J|_{\Pi_i}=0$).
We will call such almost complex structure $J$ on 8-dimensional manifold $M$
\emph{transversally non-degenerate}.

This assumption imposes topological restrictions on $M$ and the homotopy type of $J$.
Indeed, 
$N_J:\Pi_1\ot\Pi_2\to T$ is an anti-isomorphism and therefore
the total Chern class of the tangent bundle is equal to
 $$
c(T,J)=c(\Pi_1\oplus\Pi_2)=c(\overline{\Pi_1\ot\Pi_2})
 $$
Denoting $c_k'=c_k(\Pi_1)$, $c_k''=c_k(\Pi_2)$ the Chern classes of the
complex rank 2 bundles $\Pi_i$ we can use the formulae $c(\Pi_1\ot\Pi_2)=
1+2(c_1'+c_1'')+((c_1'+c_1'')^2+c_1'c_1''+2(c_2'+c_2''))+((c_1'+c_1'')(c_1'c_1''+2(c_2'+c_2''))
+((c_2'-c_2'')^2+(c_1'+c_1'')(c_1'c_2''+c_1''c_2'))$ and
$c(T,J)=(1+c_1'+c_2')(1+c_1''+c_2'')$. They imply the following restrictions
for the Chern classes $c_i=c_i(T,J)$ of transversally non-degenerate $J$:
 $$
3c_1=0,\quad 3c_2=-3q^2,\quad 3c_3=0,\quad 15c_4=0
 $$
for some $q\in H^4(M)$ (here $q=c_2(\Pi_1)$). In particular, $\chi(M)=0$.

Moreover generically the 6-dimensional manifold
$\Lambda^6=\{N_J(\xi,\eta):\xi\in\Pi_1,\eta\in\Pi_2\}\subset T$ meets $\Pi_i$
at two different complex lines $L^1_i,L^2_i$ (degree of $\Lambda^6$ is 2).
This means that there exist two canonical decompositions into the sum of complex lines
$\Pi_i=V_i^{1s}\oplus V_i^{2s}$ such that $N_J(V_1^{js},V_2^{js})=L_s^j$, $1\le i,j,s\le2$
(the line bundles are defined up to the changes $(i,s)\mapsto(i+1,s+1)$ and $j\mapsto j+1$ when
we consider $i,j,s\mod2$).

If the subbundles $V_i^{js}\subset\Pi_i$ are transversal to one of the lines $L_i^j$ (say to $L_i^2$),
we call $J$ (resp. $N_J$) \emph{strongly non-degenerate}.
It is easy to classify strongly non-degenerate Nijenhuis tensors
(the moduli space has complex dimension 8). Indeed, there exists a basis $e_i^j\in L_i^j$
(these 4 vectors in $T$ are defined up to simultaneous multiplication by $\sqrt[3]{1}$)
such that $N_J$ is given by 8 complex constants $\l_{i}^{js}$ and the relations
 $$
N_J(e_1^1+\l_{1}^{js}e_1^2,e_2^1+\l_{2}^{js}e_2^2)=e_s^j,\qquad e_i^1+\l_{i}^{js}e_i^2\in V_i^{js}.
 $$

Existence of strongly non-degenerate structures on $M$ implies even stronger topological
obstructions on the homotopy class of $J$.
Let for simplicity the line bundles be numerated (elsewise the torsion order has to
be multiplied by 4). Then the above relations imply that the 1st Chern classes
$c_1(V_i^{js})=c_1(L_i^j)=\a$ are independent of indices and $3\a=0$ (the claim that
$c_1(V_i^{js})=c_1(L_i^j)$ depend only on the index $i$ follows from
$\La^2\Pi_i=L_i^1\ot L_i^2=V_i^{1s}\ot V_i^{2s}$;
the rest follows from the defining equation for $N_J$).

Consequently, we have:
$c_1=4\a=\a$, $c_2=6\a^2=0$, $c_3=4\a^3=c_1^3$, $c_4=\a^4=c_1^4$
(if $H^*(M)$ has no torsion, then all Chern numbers vanish).

According to \cite{MG} the existence of an almost complex structure in dimension 8 is equivalent
to such relations:
 \begin{gather*}
-c_1^4+4c_1^2c_2+c_1c_3+3c_2^2-c_4\equiv0\,\op{mod}720,\\
2c_1^4+c_1^2c_2\equiv0\,\op{mod}12,\quad c_1c_3-2c_4\equiv0\,\op{mod}4.
 \end{gather*}
For strongly non-degenerate $J$ these relations specify to the only constraint $c_4\equiv0\,\op{mod}720$.
It would be interesting to understand the restrictions on Chern classes if $N_J$
is non-degenerate in the weak sense, $\op{Im}(N_J)=T$.

\appendix

\section{Associated Hermitian metrics in dimension 6}\label{SA}

 \abz
In \cite{B} an invariant Hermitian metric was associated to an almost complex structure in dimension 6.
This construction depends pointwise on $J,N_J$, and can be written following \cite{K$_5$} so
(the trace is over $\R$):
 $$
h(\xi,\eta)=\op{Tr}[N_J(\xi,N_J(\eta,\cdot))+N_J(\eta,N_J(\x,\cdot))].
 $$
Since it satisfies the property $h(J\xi,J\eta)=h(\xi,\eta)$, it has (complex) type $(1,1)$.

The corresponding almost symplectic form equals $\oo(\xi,\eta)=h(J\x,\eta)$
(generically it is not closed\footnote{In the most symmetric non-degenerate cases the
3-form $d\oo$ is the restriction of a generic 3-form in $\R^7$ with stabilizer $G_2$.}).
The hermitian metric and the orientation on $M^6$ induced by $J$ determine the smooth volume form
$\Omega$, and there is also the canonical holomorphic $(3,0)$-form $\sigma$ related to $\Omega$
 $$
\Omega=\frac{i}4\sigma\wedge\bar\sigma=\pm\frac13\omega^3
 $$
(the form $\sigma$ is normalization and alternation of the following $\C$-valued 3-tensor:
$\varsigma(X,Y,Z)=h(N_J(X,Y),Z)-i\,h(N_J(X,Y),JZ)$\,).

For the cases of our current interests the metric on $\m=T_xM$ in complex coordinates
(naturally related to the basis in which the normal form is given) is the following:
 \begin{equation}\label{hstr12}
h=dz_1\cdot d\bar z_1+dz_2\cdot d\bar z_2+\e\,dz_3\cdot d\bar z_3,
 \end{equation}
where $\e=+1$ for (\ref{NeqS3}) and $\e=-1$ for (\ref{NeqS2}).

The almost complex structure on $\m$ writes so: $J=J_1+J_2+J_3$, where
$J_k=i\,\partial_{z_k}\otimes dz_k-i\,\bar\partial_{z_k}\otimes d\bar z_k=
\partial_{y_k}\otimes dx_k-\partial_{x_k}\otimes dy_k$ is the complex
structure on the $k$-th summand in $\C^3$.

The canonical holomorphic and the smooth volume forms in both cases are
 $$
\sigma=d^3z=dz_1\wedge dz_2\wedge dz_3,\qquad \Omega=\frac{i}4d^3z\wedge d^3\bar z.
 $$
Now it's easy to calculate the stabilizer $\h=\g_1$, which is the automorphism of the triple
$(h,J,\z)$ on $\m$, in both cases. It is important that
this triple has the same automorphism group as the pair $(J,N_J)$ \cite{B,K$_5$}.

We have $\h=\op{su}(3)$ for (\ref{NeqS3}) and $\h=\op{su}(2,1)$ for (\ref{NeqS2}).

\section{Lie algebras in dimension 14}\label{SB}

Here we consider the stabilizer $\h=\g_1$ of dimension 8 from the previous section naturally acting on $\m=\g_0=T_x$ of dimension 6. The symmetry algebra has maximal dimension 14
if we can construct Lie algebra $\g$ allowing the following exact 3-sequence
 $$
0\to\h\to\g\to\m\to0.
 $$
There is indeed a flat solution $\g=\h\ltimes\m$, but then $N_J=0$ that is not appropriate.
Otherwise we will show that the Lie algebra $\g$ can be reconstructed uniquely
from the above exact 3-sequence for both (\ref{NeqS3}) and (\ref{NeqS2}).

Since $\h$ is simple, we have direct decomposition of $\h$-representations $\g=\h\oplus\m$,
and $\m=\C^3$ as $\h$-module.
Thus to restore the Lie algebra $\g$ we need two components of the $\h$-morphisms $\Lambda^2\m\to\g$.
The $\m$-part is given uniquely (up to scale, see below), so we have to determine
the $\h$-component of $[\m,\m]$.

As $\h$-representation $\Lambda^2\m$ is reducible, and it decomposes into three submodules of dimensions 6, 1, 8, which are better described in the complexification:
 $$
\Lambda^2_\C\m=\Lambda^{2,0}(\m)\oplus\Lambda^{1,1}(\m)\oplus\Lambda^{0,2}(\m).
 $$
The piece $\Lambda^{2,0}\oplus\Lambda^{0,2}$ is the complexification of a 6 dimensional summand, and
$\Lambda^{1,1}=\C\oplus\Lambda^{1,1}_0$ corresponds to $1+8$ dimensional piece.
Here we identify via the Hermitian metric $\Lambda^{1,1}=\op{End}(\m)_\C$ and
the endomorphisms $\h$-equivariantly split into the scalar and traceless parts
(now the trace is over $\C$).

Thus the $\m$-part of the above $\h$-morphism $\Lambda^2\m\to\m$ is given by the projection
($\C$-antilinearity yields the required tensor $N_J\ne0$)
 $$
\Lambda^2_\C\m\supset\Lambda^{2,0}(\m)\oplus\Lambda^{0,2}(\m)
\simeq\Lambda^{0,1}(\m)\oplus\Lambda^{1,0}(\m)=\m_\C
 $$
Formula of this morphism is (the scaling real factor $a$ is to be fixed later)
 \begin{equation}\label{a}
\xi\we\eta\mapsto a\cdot\sigma(\xi,\eta,\cdot)^{\sharp_h}.
 \end{equation}

In the complex basis of $\m_\C=\m^{1,0}+\m^{0,1}$ the $\m$-component of the
is given by the relations ($\e=+1$ for (\ref{NeqS3}) and $\e=-1$ for (\ref{NeqS2}))
 $$
[\partial_{z_1},\partial_{z_2}]=a\e\,\bar\partial_{z_3},\ \
[\partial_{z_2},\partial_{z_3}]=a\,\bar\partial_{z_1},\ \
[\partial_{z_3},\partial_{z_1}]=a\,\bar\partial_{z_2}
 $$
 and their conjugates.

The $\h$-part $\Lambda^2\m\to\h$ is given via the complexification as the projection
 $$
\Lambda^2_\C\m\supset\Lambda^{1,1}\to\Lambda^{1,1}_0\simeq\h_\C,
 $$
where the identification $\Lambda^{1,1}=\m_{0,1}\circ\m_{1,0}\simeq\op{End}(\m)_\C\subset
\m^{1,0}\otimes\m_{1,0}+\m^{0,1}\otimes\m_{0,1}$
is given by the Hermitian metric, $\m_{1,0}^*=\m^{0,1}$ and $\m_{0,1}^*=\m^{1,0}$.
Formula of this morphism is (the scaling real factor $b$ is to be fixed later)
 \begin{equation}\label{b}
\xi\we\eta\mapsto b\cdot(3\xi\ot h(\eta,\cdot)-3\eta\ot h(\xi,\cdot)+\omega(\xi,\eta)\,J).
 \end{equation}

In the complex basis the $\h$-component of the bracket on $\m$ is given by
 $$
[\partial_{z_1},\bar\partial_{z_1}]=ib\,(J-3J_1),\ \
[\partial_{z_2},\bar\partial_{z_2}]=ib\,(J-3J_2),\ \
[\partial_{z_3},\bar\partial_{z_3}]=ib\e\,(J-3J_3),
 $$
 \vskip-27pt
 \begin{alignat*}{2}
[\partial_{z_1},\bar\partial_{z_2}]=3b\,(\partial_{z_1}\otimes dz_2-\bar\partial_{z_2}\otimes d\bar z_1), & \quad
[\partial_{z_1},\bar\partial_{z_3}]=3b\,(\e\,\partial_{z_1}\otimes dz_3-\bar\partial_{z_3}\otimes d\bar z_1), \\
[\partial_{z_2},\bar\partial_{z_1}]=3b\,(\partial_{z_2}\otimes dz_1-\bar\partial_{z_1}\otimes d\bar z_2), & \quad
[\partial_{z_2},\bar\partial_{z_3}]=3b\,(\e\,\partial_{z_2}\otimes dz_3-\bar\partial_{z_3}\otimes d\bar z_2), \\
[\partial_{z_3},\bar\partial_{z_1}]=3b\,(\partial_{z_3}\otimes dz_1-\e\,\bar\partial_{z_1}\otimes d\bar z_3),\!\!\! & \quad
[\partial_{z_3},\bar\partial_{z_2}]=3b\,(\partial_{z_3}\otimes dz_2-\e\,\bar\partial_{z_2}\otimes d\bar z_3).
 \end{alignat*}
and the conjugates of these (notice that we study the complexification of the real Lie
algebra, so commutator relations can be written in the real basis of $\m$).

Let us check the Jacobi identity (we can use complex basis for this). If all 3 vectors have $(1,0)$-type,
the only non-trivial relation is
 $$
[\partial_{z_1},[\partial_{z_2},\partial_{z_3}]]+
[\partial_{z_2},[\partial_{z_3},\partial_{z_1}]]+
[\partial_{z_3},[\partial_{z_1},\partial_{z_2}]]=0,
 $$
which is equivalent to the trace-zero condition
 $$
[\partial_{z_1},\bar\partial_{z_1}]+[\partial_{z_2},\bar\partial_{z_2}]
+\e\,[\partial_{z_3},\bar\partial_{z_3}]=0.
 $$

If we use two $(1,0)$-types and one $(0,1)$-type, then there are 3 similar
relations with 3 different indices, like
 $$
[\bar\partial_{z_1},[\partial_{z_2},\partial_{z_3}]]+
[\partial_{z_2},[\partial_{z_3},\bar\partial_{z_1}]]+
[\partial_{z_3},[\bar\partial_{z_1},\partial_{z_2}]]=0+0+0=0
 $$
and 6 similar relations with 2 indices equal and 1 different, like
 $$
[\bar\partial_{z_1},[\partial_{z_1},\partial_{z_2}]]+
[\partial_{z_1},[\partial_{z_2},\bar\partial_{z_1}]]+
[\partial_{z_2},[\bar\partial_{z_1},\partial_{z_1}]]=0.
 $$
This latter writes in details as follows
 $$
a\e\,[\bar\partial_{z_1},\bar\partial_{z_3}]-
3b\,(\partial_{z_2}\otimes dz_1-\bar\partial_{z_1}\otimes d\bar z_2)(\partial_{z_1})+
ib\,(J-3J_1)\partial_{z_2}=-(4b+a^2\e)\partial_{z_2}=0.
 $$
It is the same factor $(4b+a^2\e)$ in all 6 relations, and the other relations are conjugated
to these, so $\g$ has the unique Lie algebra structure iff $4b+a^2\e=0$.
In the non-flat case $(a,b)\neq(0,0)$ we can re-scale the constants to
$a=\pm2$, $b=-\e$, and the choice of sign is not essential.

Thus for any of the two cases $\e=\pm1$ we get a unique solution $a=+2,b=-\e$,
and, exploiting the Levi decomposition, we conclude that the Lie algebra $\g$ is simple.
A straightforward check shows that $\e=+1$ corresponds to the compact real version of the simple
exceptional Lie algebra $\g=\op{Lie}(G_2^c)$, while $\e=-1$ corresponds to the split real version
$\g=\op{Lie}(G_2^n)$.

\medskip

\textsc{Acknowledgment.} I would like to thank D.\,V. Alekseevsky, T. Wilse and H. Winter
for useful discussions on the topics from representation theory.

\footnotesize

 \vspace{-5pt} \hspace{-20pt} {\hbox to 12cm{ \hrulefill }}
\vspace{-1pt}

{\footnotesize \hspace{-10pt} Institute of Mathematics and
Statistics, University of Troms\o, Troms\o\ 90-37, Norway.

\hspace{-10pt} E-mail: \quad boris.kruglikov\verb"@"uit.no} \vspace{-1pt}

\end{document}